\documentclass{gLMA2e}

\usepackage{amsmath, amsthm, amssymb,amsfonts}					
\usepackage{hyperref}
\usepackage{graphicx}
\usepackage[usenames,dvipsnames,svgnames,table]{xcolor}
\usepackage{enumitem}
\usepackage{multirow}
\usepackage{float}

\usepackage{pgfplots}

\everymath{\displaystyle}

\newtheorem{thm}{Theorem}[section]
\newtheorem{lem}[thm]{Lemma}
\newtheorem{cor}[thm]{Corollary}

\theoremstyle{definition}			                						
\newtheorem{mydef}[thm]{Definition}
\newtheorem{rem}[thm]{Remark}
\newtheorem{ex}[thm]{Example}

\numberwithin{equation}{section}		            					

\newcommand{\bb}[1]{\mathbb{#1}}								
\newcommand{\mat}[2]{{M}_{#1}(#2)}							
\newcommand{\gl}[2]{{GL}_{#1}(#2)}								

\newcommand{\inv}[1]{#1^{-1}}								

\newcommand{\diag}[1]{\operatorname{\rm diag}\left( #1 \right)}			

\newcommand{\coni}[1]{\operatorname{\rm coni}\left( #1 \right)}			

\newcommand{\cone}[1]{\mathcal{C}\left( #1 \right)}					
\newcommand{\rc}[1]{\mathcal{C}_r\left( #1 \right)}					

\newcommand{\hyp}[2]{#1 \hyperref[#2]{\ref*{#2}}}					

\begin{document}

\title{Row Cones, Perron Similarities, and Nonnegative Spectra}
\author{
	\name{Charles R. Johnson\textsuperscript{a} and Pietro Paparella\textsuperscript{b}$^{\ast}$\thanks{$^\ast$Corresponding author. Email: pietrop@uw.edu}}
	\affil{\textsuperscript{a}Department of Mathematics, College of William and Mary, Williamsburg, VA 23187-8795, USA
	\textsuperscript{b}Division of Engineering and Mathematics, University of Washington Bothell, Bothell, WA 98011-8246, USA}}

\maketitle

\begin{abstract}
In further pursuit of the diagonalizable \emph{real nonnegative inverse eigenvalue problem} (RNIEP), we study the relationship between the \emph{row cone} $\rc{S}$ and the \emph{spectracone} $\cone{S}$ of a Perron similarity $S$. In the process, a new kind of matrix, \emph{row Hadamard conic} (RHC), is defined and related to the D-RNIEP. Characterizations are given when $\rc{S} = \cone{S}$, and explicit examples are given for all possible set-theoretic relationships between the two cones. The symmetric NIEP is the special case of the D-RNIEP in which the Perron similarity $S$ is also orthogonal.
\end{abstract}

\begin{keywords} 
Perron similarity, Row Hadamard conic, Real nonnegative inverse eigenvalue problem, Hadamard product, Row cone, Spectra-cone
\end{keywords}

\begin{classcode}
Primary 15A18, 15A29, 15B48
\end{classcode}

\date{\today}

\pagenumbering{arabic}

\section{Introduction}

In \cite{jp2016}, the notion of a \emph{Perron similarity} was introduced and explored. Here, by a Perron similarity, we mean a matrix $S \in \gl{n}{\bb{R}}$ such that there is a real, diagonal, nonscalar matrix $D$ with $S D \inv{S} \geq 0$ (all inequalities in this work apply to all entries of a matrix or vector). Of course, the diagonal entries of $D$ then constitute a particular instance of the \emph{real nonnegative inverse eigenvalue problem} (RNIEP). The RNIEP is a variant of the \emph{nonegative inverse eigenvalue problem} (NIEP), one of the premier unsolved problems in matrix analysis (see, e.g., \cite{c1998}, \cite{eln2004}, or \cite{m1988}).

Perron similarities were characterized in several ways in \cite{jp2016} and it was shown that
\[ \cone{S}:= \{ x \in \bb{R}^n: S \diag{x} \inv{S} \geq 0 \} \]
is a \emph{polyhedral cone}, i.e., a convex cone in $\bb{R}^n$ with finitely-many extremals. This was called the \emph{spectracone of $S$}, and a certain cross-section, a polytope called the \emph{spectratope of $S$}, was also discussed. Of course, taken together, the spectracones of all Perron similarities constitute the solution to the \emph{diagonalizable RNIEP} (D-RNIEP), another challenging variant that is also unsolved.

In \cite{jp2016}, it was noted that $\cone{S}$ can coincide with $\rc{S}$, the convex cone generated by the rows of $S$. This occurs, for example, for a certain special family of \emph{Hadamard matrices} called \emph{Walsh matrices}. Obviously, when there is a clear relationship with $\rc{S}$, it helps to determine $\cone{S}$. Our purpose here is to identify the possible relationships between $\cone{S}$ and $\rc{S}$, and, where possible, to characterize how they occur. This is a natural step in understanding more deeply Perron similarities and their role in the D-RNIEP. Without explicitly giving the statements, we note that the symmetric NIEP (SNIEP) is the special case of our results in which the Perron similarity $S$ is orthogonal.

In the next section, we make and record some simple observations that underlay this work. Then, in section 3, we give sufficient conditions and characterizations of containment in either direction between the cones $\rc{S}$ and $\cone{S}$. This facilitates characterization of when the two cones are equal. In section 4 are indicated examples of the remaining possible set-theoretic relationships between the two cones $\rc{S}$ and $\cone{S}$. In the course of this work, a new kind of matrix, that we call \emph{row Hadamard conic} (RHC), naturally arises, and we suspect that these may be of interest on their own.

\section{Useful Observations}

A simple fact about row-cones is the following.

\begin{lem}
\label{lemtwoone}
For $S \in \gl{n}{\bb{R}}$, the row-vector $x^\top \in \rc{S}$ if and only if $x^\top \inv{S} \geq 0$. 
\end{lem}

\begin{proof}
Of course, $x^\top \in \rc{S}$ if and only if $x^\top = y^\top S$ for some $y \in \bb{R}$ with $y \geq 0$. Since $S$ is invertible, $x^\top \inv{S} = y^\top \geq 0 $ is unique.
\end{proof}

For a given matrix $A$, we denote the $i$\textsuperscript{th}-row of $A$ by $r_i(A)$ (when the context is clear, $r_i(A)$ is abbreviated to $r_i$). Denote by $e$ and $e_i$ the all-ones vector and the $i$\textsuperscript{th} canonical basis vector, respectively. Recall that $\circ$ denotes the Hadamard (entrywise) product of two vectors or matrices of the same size. 

A key observation for our work is the following.

\begin{lem}
\label{lemtwotwo}
Let $S \in \gl{n}{\bb{R}}$. If $R_i := \diag{r_i}$,  then $S R_i \inv{S} \geq 0$, $\forall i \in \{ 1, \dots, n \}$, if and only if $r_i \circ r_j \in \rc{S}$, for every $j \in \{1, \dots, n\}$.
\end{lem}

\begin{proof}
View $S R_i \inv{S}$ as $(S R_i) \inv{S}$ and note that the rows of $S R_i$ are the Hadamard products $r_i \circ r_j$, $j \in \{1, \dots, n\}$. But, each matrix $S R_i \inv{S}$ is nonnegative if and only if each row of it is nonnegative. The rows of $S R_i \inv{S}$ are just the rows of $S R_i$, each multiplied into $\inv{S}$. This means that $(r_i \circ r_j) \inv{S} \geq 0$, for $j=1,\dots,n$, and by \hyp{Lemma}{lemtwoone} that $r_i \circ r_j \in \rc{S}$, $j = 1, \dots,n$.
\end{proof}

The above leads naturally to the following important definition of a new concept that seems interesting by itself.

\begin{mydef}
The matrix $S \in \mat{m,n}{\bb{R}}$ is called \emph{row Hadamard conic} (RHC) if 
\[ r_i \circ r_j \in \rc{S},~\forall i,j \in \{1,\dots,m\}. \]
\end{mydef}

\begin{rem}
Since $\rc{S}$ is the cone generated by the rows of $S$, it follows quickly from the definition that the Hadamard product of any number of rows (repeats allowed) of an RHC matrix $S$ also lies in $\rc{S}$.
\end{rem}

\begin{ex}
Of course, $S$ is RHC if the rows themselves are closed under the Hadamard product. This does happen for certain Hadamard matrices (see \cite{jp2016} for a discussion and implications); invertible van der Monde matrices that diagonalize nonnegative companion matrices; and all permutation matrices.

As a particular example of a RHC matrix, we appeal to a well-known example from the literature: recall that a finite list of real numbers is called a \emph{Sule\u{\i}manova spectrum} if it contains exactly one positive value. In \cite{s1949}, it was loosely shown that every such list is the spectrum of a nonnegative matrix. Friedland \cite{f1978} and Perfect \cite{p1953} proved Sule\u{\i}manova's result via companion matrices (for other proofs, see references in \cite{f1978}). Recently, a constructive proof of this result was given in \cite{p2016}. 

Consider the list $(1,-a,-b)$, with $0 \leq a,b < 1$. The Vandermonde matrix
\[ 
S = 
\begin{bmatrix}
1 & 1 & 1 \\
1 & -a & -b \\
1 & a^2 & b^2
\end{bmatrix} \]
is RHC since every row is realizable (\hyp{Lemma}{lemtwotwo}).
\end{ex}

Both $\cone{S}$ and $\rc{S}$ are unchanged, or changed predictably, by certain basic matrix transforms of $S$. It is useful to have these available. In \hyp{Table}{tab}, $P$ denotes a permutation matrix, $D_+$ a positive diagonal matrix, and $D$ an invertible diagonal matrix. Each case is either known or straightforward to verify.
\begin{table}[H]
\centering
\begin{tabular}{c | c | c}
 & effect on $\cone{S}$ & effect on $\rc{S}$ \\
 \hline
 $PS$ & none & none \\
 $D_+ S$ & none & none \\
 $SP$ & permutation of spectra only & cone is permuted entrywise \\
 $SD$ & none & cone entries are re-signed and relatively-scaled 
 \end{tabular}
 \caption{Effects on $\cone{S}$ and $\rc{S}$ of matrix transformations.}
 \label{tab}
\end{table}

\section{Containment Results}

For a certain Perron similarity $S$, we wish to describe all possibilities for the relationship between $\rc{S}$ and $\cone{S}$, at least by example. First, we characterize containment in one direction. 

\begin{thm}
\label{thmthreeone}
If $S \in \gl{n}{\bb{R}}$ is a Perron similarity, then $\rc{S}\subseteq\cone{S}$ if and only if $S$ is RHC.
\end{thm}

\begin{proof}
Since both are cones, the claimed containment is valid if each $r_i$ lies in $\cone{S}$. The import of \hyp{Lemma}{lemtwotwo} is that this happens for a particular $r_i$ exactly when the Hadamard product condition is met for all rows with $r_i$. But RHC means that this condition is met for every $i$, completing the proof.
\end{proof}

What about containment the other way? Here, we have a sufficient condition.

\begin{thm}
\label{thmthreetwo}
Let $S \in \gl{n}{\bb{R}}$ be a Perron similarity. If some row of $S$ is $e^\top$, then $\cone{S} \subseteq \rc{S}$.
\end{thm}

\begin{proof}
Suppose that $x^\top$, viewed as a row-vector, satisfies $x^\top \in \cone{S}$, and let $X = \diag{x}$. Since $e^\top$ is a row of $S$, $x^\top \circ e = x^\top$ is a row of $SX$, and from $SX\inv{S} \geq 0$, we have $x^\top \inv{S} \geq 0$. From \hyp{Lemma}{lemtwoone}, this means that $x^\top \in \rc{S}$, as was to be shown.
\end{proof}

\begin{rem}
\label{remthreethree}
More generally, but less concretely, we see in the same way that $\cone{S} \subseteq \rc{S}$ if and only if every spectral vector $x^\top \in \cone{S}$ satisfies $x^\top \inv{S} \geq 0$, or, equivalently, each of the finite number of extremals $y^\top$ of $\cone{S}$ satisfies $y^\top \inv{S} \geq 0$.
\end{rem}

\begin{rem}
It follows from \hyp{Theorem}{thmthreeone} that if $x^\top$ occurs as a row of a Perron similarity that is RHC, then the components of $x$ form a nonnegative spectrum. This may provide a method to verify that a spectrum is nonnegative. We do not know if this happens for every nonnegative spectrum.
\end{rem}

With a weaker condition on $S$, we have, as a consequence, containment of $\cone{S}$ in a slightly different cone. As usual, denote by $R_i$ the diagonal matrix $\diag{r_i}$.

\begin{cor}
\label{corthreefour}
Let $S \in \gl{n}{\bb{R}}$ be a Perron similarity and suppose that $r_i$ is a totally nonzero row of $S$. Then $\cone{S} \subseteq \rc{S \inv{R_i}}$. 
\end{cor}

\begin{proof}
Since $r_i(S \inv{R_i}) = e^\top$, we may apply \hyp{Theorem}{thmthreetwo} to $S \inv{R_i}$. But $\cone{S} = \cone{S \inv{R_i}}$, while $\rc{S \inv{R_i}}$ need not be $\rc{S}$.
\end{proof}

\begin{rem}
Of course, $S$ may have several nonnzero rows, and \hyp{Corollary}{corthreefour} may be applied to each to narow the containment of $\cone{S}$. Thus, 
\[ \cone{S} \subseteq \bigcap_{i:r_i~\text{is totally nonzero}} \rc{S\inv{R_i}}. \]
\end{rem}

We may now describe when $\cone{S} = \rc{S}$. 

\begin{thm}
Let $S \in \gl{n}{\bb{R}}$ be a Perron similarity and suppose that $r_i = e^\top$ for some $i \in \{1, \dots, n\}$. Then $\cone{S} = \rc{S}$ if and only if $S$ is RHC.
\end{thm}

\begin{proof}
Combine the containments in Theorems \hyperref[thmthreeone]{\ref*{thmthreeone}} and \hyperref[thmthreetwo]{\ref*{thmthreetwo}}.
\end{proof}

Alternately, we have, based upon \hyp{Lemma}{lemtwotwo}, the following.
\begin{thm}
Let $S \in \gl{n}{\bb{R}}$ be a Perron similarity and suppose that $r_i = e^\top$ for some $i \in \{1, \dots, n\}$. Then $\cone{S} = \rc{S}$ if and only if $r_i \in \cone{S}$ for every $i \in \{1,\dots,n\}$.
\end{thm}

Finally, we have, based upon \hyp{Remark}{remthreethree}, the following.

\begin{thm}
Let $S \in \gl{n}{\bb{R}}$ be a Perron similarity. Then $\cone{S} = \rc{S}$ if and only if $S$ is RHC and the extremals $y$ of $\cone{S}$ satisfy $y^\top \inv{S} \geq 0$.
\end{thm}

\section{Possible Relationships between $\cone{S}$ and $\rc{S}$}

Based upon the results of the prior section, our two cones may be equal or one contained in the other. And since the indicated conditions can be met, it is clear that these may occur. But, we wish to indicate here, by example, what else may occur. 

Denote by $\coni{v_1,\dots,v_p}$ the \emph{conical hull} of $v_1, \dots, v_p \in \bb{R}^n$, i.e., 
\[ \coni{v_1,\dots,v_p} = \left\{ \sum_{i=1}^ p \alpha_i v_i : \alpha_i \geq 0,~\forall i \in \{1,\dots,p\} \right\}. \]

\begin{ex}
Each of $\rc{S}$ and $\cone{S}$ may be properly contained in the other.
\begin{enumerate}[label=\alph*)]
\item If 
\[ S= \begin{bmatrix} 1 & 1 \\ 0 & -1 \end{bmatrix}~\text{and}~D=\begin{bmatrix} x & 0 \\ 0 & y \end{bmatrix}, \]
then 
\[ A := SD\inv{S} = \begin{bmatrix} x & x-y \\ 0 & y \end{bmatrix}, \] 
and $A \geq 0$ if and only if $x$, $y \geq 0$ and $y \leq x$. Thus, $\cone{S} = \coni{e,e_1} \subset \coni{e,-e_2} = \rc{S}$.

\item If 
\[ S= \begin{bmatrix} 1 & .1 \\ 1 & -.1 \end{bmatrix}~\text{and}~D=\begin{bmatrix} x & 0 \\ 0 & y \end{bmatrix}, \]
then 
\[ A := SD\inv{S} = \frac{1}{2} \begin{bmatrix} x+y & x-y \\ x-y & x+y \end{bmatrix}, \] 
and $A \geq 0$ if and only if $y \leq -x$ and $y \leq x$. Thus, $\rc{S} \subset \coni{e,[1,-1]^\top} = \cone{S}$.
\end{enumerate}
\end{ex}

\begin{ex}
The two cones intersect at the origin: if 
\[ S= \begin{bmatrix} -1 & -1 \\ -1 & 0 \end{bmatrix}~\text{and}~D=\begin{bmatrix} x & 0 \\ 0 & y \end{bmatrix}, \]
then 
\[ A := SD\inv{S} = \begin{bmatrix} y & x-y \\ 0 & x \end{bmatrix}, \] 
and $A \geq 0$ if and only if $x$, $y \geq 0$ and $y \leq x$. Thus, $\cone{S} = \coni{e,e_1}$, $\rc{S} = \coni{-e,-e_1}$, and $\cone{S} \cap \rc{S} = \{ 0 \}$.
\end{ex}

\begin{rem}
It is interesting to note that the matrices 
\[ \begin{bmatrix} 1 & 1 \\ 1 & 0 \end{bmatrix}~\text{and}~\begin{bmatrix} 1 & 1 \\ 0 & -1 \end{bmatrix} \] 
have the same spectracones, yet one cannot be obtained from the other via the operations listed in \hyp{Table}{tab}.
\end{rem}

\begin{ex} \label{ex44}
The two cones may be noncomparable, though intersecting. Consider the spectracone of  
 \[ S= \begin{bmatrix} 1 & 1 & 0\\ 1 & -.5 & 1 \\ 1 & -.5 & -1 \end{bmatrix}. \]
 A straightforward, but tedious, calculation reveals that $\cone{S} = \coni{v_1, v_2, v_3, v_4}$, with
 \[ v_1 =  \begin{bmatrix} 1 \\ 1 \\ 1 \end{bmatrix},~
 v_2 =  \begin{bmatrix} 1 \\ -.5 \\ .5 \end{bmatrix},~
 v_3 =  \begin{bmatrix} 1 \\ -.5 \\ -.5 \end{bmatrix},~\text{and}~ 
 v_4 =  \begin{bmatrix} 1 \\ 1 \\ -1 \end{bmatrix}. \]
Thus, $\rc{S} \not \subseteq \cone{S}$ and $\cone{S} \not \subseteq \rc{S}$. The cones $\rc{S}$ and $\cone{S}$ are depicted in \hyp{Figure}{spec}.
\begin{figure}[H]
\centering
\begin{tikzpicture}
\begin{axis}
[
view={55}{30},
xmin=0,
xmax=1.05,
ymin=-1.05,
ymax=1.05,
zmin=-1.05,
zmax=1.05,
z label style={rotate=-90},
xlabel=$x_1$,
ylabel=$x_2$,
zlabel=$x_3$]

\addplot3[dotted, thick,fill=orange,opacity=.5] coordinates{(1,1,1) (1,-.5,.5) (1,-.5,-.5) (1,1,-1) (1,1,1)};
\addplot3[dotted, thick,fill=orange,opacity=.2] coordinates{(0,0,0) (1,-.5,.5) (1,-.5,-.5) (0,0,0)};
\addplot3[dotted, thick,fill=orange,opacity=.2] coordinates{(0,0,0) (1,-.5,.5) (1,1,1) (0,0,0)};
\addplot3[dotted, thick,fill=orange,opacity=.2] coordinates{(0,0,0) (1,-.5,-.5) (1,1,-1) (0,0,0)};

\addplot3[thick,fill=red,opacity=.3] coordinates{(0,0,0) (1,-.5,1) (1,-.5,-1) (0,0,0)}; 
\addplot3[thick,fill=red,opacity=.1] coordinates{(0,0,0) (1,1,0) (1,-.5,1) (0,0,0)};
\addplot3[thick,fill=red,opacity=.5] coordinates{(1,1,0) (1,-.5,1) (1,-.5,-1)(1,1,0)};

\end{axis}
\end{tikzpicture}
\caption{The cones $\rc{S}$ (red) and $\cone{S}$ (orange) from \hyp{Example}{ex44}.}
\label{spec}
\end{figure}
\end{ex}

\bibliographystyle{gLMA}
\bibliography{master}

\end{document}